\documentclass{amsart}

\usepackage[colorlinks=true,urlcolor=blue, citecolor=red,linkcolor=blue,linktocpage,pdfpagelabels, bookmarksnumbered,bookmarksopen]{hyperref}
\usepackage{amsthm} 
\usepackage{latexsym,amsmath,amssymb}

\usepackage{accents}
\usepackage{esint}
\usepackage{mathtools} 
\usepackage{xparse} 
\usepackage[capitalize]{cleveref}

\usepackage[shortlabels]{enumitem}

\usepackage[textsize=small]{todonotes}
\title[Tangential Calderon-Zygmund estimates]{Global tangential Calderon-Zygmund type estimates for the regional fractional Laplacian}

\author{S. Khomrutai}
\address[Sujin Khomrutai]{Department of Mathematics and Computer Science, Faculty of Science, Chulalongkorn University, Bangkok 10330, Thailand}
\email{sujin.k@chula.ac.th}

\author{A. Schikorra}
\address[Armin Schikorra]{Department of Mathematics,
University of Pittsburgh,
301 Thackeray Hall,
Pittsburgh, PA 15260, USA}
\email{armin@pitt.edu}

\author{A. Seesanea}
\address[Adisak Seesanea]{Sirindhorn International Institute of Technology, Thammasat University, Pathum Thani 12120,Thailand}
\email{adisak.see@siit.tu.ac.th}

\author{S. Yeepo}
\address[Sasikarn Yeepo]{Sirindhorn International Institute of Technology, Thammasat University, Pathum Thani 12120,Thailand}
\email{sasikarn.y@g.siit.tu.ac.th}

\newcommand{\N}{{\mathbb N}}

\newtheorem{theorem}{Theorem}
\newtheorem{lemma}[theorem]{Lemma}

\newtheorem{proposition}[theorem]{Proposition}

\theoremstyle{definition}

\theoremstyle{remark}

\newcommand\supp{{\rm supp\,}}

\newcommand{\R}{\mathbb{R}}

\newcommand{\brac}[1]{\left (#1 \right )}
\newcommand{\abs}[1]{\left\lvert #1 \right \rvert}

\newcommand{\sm}{\setminus}

\newcommand{\barint}{
\rule[.036in]{.12in}{.009in}\kern-.16in \displaystyle\int }

\newcommand{\barcal}{\text{$ \rule[.036in]{.11in}{.007in}\kern-.128in\int $}}

\def\mvint_#1{\mathchoice
          {\mathop{\vrule width 6pt height 3 pt depth -2.5pt
                  \kern -8pt \intop}\nolimits_{\kern -3pt #1}}%
          {\mathop{\vrule width 5pt height 3 pt depth -2.6pt
                  \kern -6pt \intop}\nolimits_{#1}}%
          {\mathop{\vrule width 5pt height 3 pt depth -2.6pt
                  \kern -6pt \intop}\nolimits_{#1}}%
          {\mathop{\vrule width 5pt height 3 pt depth -2.6pt
                  \kern -6pt \intop}\nolimits_{#1}}}

\numberwithin{theorem}{section} \numberwithin{equation}{section}

\newcommand{\lap}{\Delta }
\newcommand{\aleq}{\lesssim}

\newcommand{\aeq}{\approx}

\newcommand{\laps}[1]{(-\Delta)^{\frac{#1}{2}}}
\newcommand{\Ds}[1]{|D|^{#1}}
\newcommand{\Dels}[1]{(-\Delta)^{#1}}

\newcommand{\lapms}[1]{\mathcal{I}_{#1}}

\usepackage{scalerel}[2014/03/10]
\usepackage[usestackEOL]{stackengine}
\def\avint{\,\ThisStyle{\ensurestackMath{%
			\stackinset{c}{.2\LMpt}{c}{.5\LMpt}{\SavedStyle-}{\SavedStyle\phantom{\int}}}%
		\setbox0=\hbox{$\SavedStyle\int\,$}\kern-\wd0}\int}

\let\latexchi\chi
\makeatletter
\renewcommand\chi{\@ifnextchar_\sub@chi\latexchi}
\newcommand{\sub@chi}[2]{
  \@ifnextchar^{\subsup@chi{#2}}{\latexchi^{}_{#2}}%
}
\newcommand{\subsup@chi}[3]{
  \latexchi_{#1}^{#3}%
}
\makeatother

\begin{document}
\begin{abstract}
We discuss tangential Sobolev-estimates up to the boundary for solutions to the regional fractional Laplacian on the upper half-plane. These estimates can be used to reduce the boundary Calderon-Zygmund theory of any dimension to a one-dimensional nonlocal problem.
\end{abstract}

\maketitle 
\tableofcontents

\section{Introduction}
In this note we are interested in global Calderon-Zygmund estimates for the regional Laplace equation
\[
 \begin{cases}
\Dels{s}_{\R^n_+} u = g \quad &\text{in $\R^n_+$},\\
u = 0 \quad &\text{on $\partial \R^{n-1} \times \{0\}$},
 \end{cases}
\]
in distributional sense, i.e. $u \in W^{s,2}_{0}(\R^n_+)$ solving 
\[
 \int_{\R^n_+}\int_{\R^n_+} \frac{(u(x)-u(y))(\varphi(x)-\varphi(y))}{|x-y|^{n+2s}}\, dx\, dy = g[\varphi] \quad \forall \varphi \in C_c^\infty(\R^n_+).
\]

Although there has been substantial progress for boundary regularity for the \emph{global} fractional Laplacian, see e.g. \cite{AFLY21,LL2022}, or \emph{H\"older continuity} up to the boundary \cite{Fall22,FRO22}, and boundary estimates for the \emph{regional} fractional Laplacian \cite{ChenHuyuan15}, for global Sobolev-space estimates up to the boundary for equations involving the regional fractional Laplacian as above the literature seems to be scarce. 

On the other hand, in particular for bootstrapping regularity for real-life nonlocal differential equation, such ``finer than Schauder''-estimates up to the boundary could play a crucial role.

One substantial issue with the fractional Laplacian is that there is no satisfying reflection principle known -- i.e. if $\laps{s}_{\R^n_+} u = f$ in $\R^n_+$, it is not clear how to extend $u$ to $\R^n_-$ to obtain a useful global equation. If one had such a reflection principle, then boundary regularity estimate would be consequence of interior regularity estimates. Another classical approach for boundary regularity is to estimate the Green's function, cf. \cite{LL2022}. But it seems that the regional fractional Laplacian has a more complicated Green's function that the global fractional Laplacian. There is also theory via Fourier analysis as developed e.g. in \cite{Grubb15} -- see also the recent lecture notes \cite{Grubb22}.

In this note we follow the spirit of a simple approach to reduce boundary regularity to a one-dimensional ODE. For the usual Laplacian it goes as follows. Assume $u$ is a solution to
\[
 \begin{cases}
  \lap u = f\quad &\text{in $\R^n_+$}\\
  u = 0 \quad &\text{on $\R^{n-1} \times \{0\}$}.
 \end{cases}
\]
By tangential differentiation, i.e. considering that for $i=1,\ldots,n-1$
\[
 \begin{cases}
  \lap \partial_{i} u = \partial_{i} f\quad \text{in $\R^n_+$}\\
  u = 0 \quad \text{on $\R^{n-1} \times \{0\}$}
 \end{cases}
\]
we naturally obtain estimates of $\nabla \partial_i u$ (e.g. from variational observations). Thus we obtain a control of $\partial_{ij} u$ where $(i,j) \neq (n,n)$.

To obtain a control of $\partial_{nn} u$ we then use the equation, namely we obtain global regularity from the (essentially one-dimensional) equation for fixed $x' \in \R^{n-1}$
\[ \partial_{nn} u(x',\cdot) = \lap u(x',\cdot) - \sum_{i=1}^{n-1} \partial_{ii} u (x',\cdot) \quad \text{in $(0,\infty)$}. \]

In principle it seems that this approach works in the nonlocal setting as well, and this is the main point of our note. However, not surprisingly, for the case of a fractional Laplacian the one-dimensional situation is not an ODE in normal direction, but we obtain a nonlocal one-dimensional differential equation -- which, for now, we cannot treat further. For simplicity, we shall only focus on the model equation in the upper half-plane, which can be easily extended to more general situations. Our first main result is the following tangential estimate.

\begin{theorem}[Tangential estimate]\label{th:tangential}
Let $s > \frac{1}{2}$ and assume $p \in (2,\infty)$ and $t \in [0,s)$.

Assume $u \in W^{s,2}_0(\R^n_+)$ satisfies the following equation in distributional sense
\begin{equation}\label{eq:ourpde}
 \Dels{s}_{\R^n_+} u = g \quad \text{in $\R^n_+$},
\end{equation}
for $g \in H^{-t,p}(\R^n)$, that is
\[
 g[\varphi] = \int G \Ds{t} \varphi \quad \text{for all $\varphi \in C_c^\infty(\R^n)$}
\]
for some $G \in L^p(\R^n)$.

Then if $u \in L^p(\R^n_+)$ then for any $\tilde{s} \in (s,1)$, $\tilde{s} < 2s-t$, we have
\[
 [u]_{W_T^{\tilde{s},p}(\R^n_+)} \aleq \|G\|_{L^p(\R^n)}
\]
where
\[
\begin{split}
[u]_{W_T^{\tilde{s},p}(\R^n_+)} :=& \brac{\int_{\R^n_+} \int_{\R^n_+} \frac{|u(x',{x_n})-u(y',{x_n})|^p}{|x-y|^{n+\tilde{s}p}} dx' dx_n dy' dy_n}^{\frac{1}{p}}\\
\equiv&c\brac{\int_{0}^\infty \int_{\R^{n-1}} \int_{\R^{n-1}} \frac{|u(x',{x_n})-u(y',{x_n})|^p}{|x'-y'|^{n-1+\tilde{s}p}} dx' dy'  dx_n}^{\frac{1}{p}}.\\
\end{split}
\]
\end{theorem}
Here $\laps{s}_{\R^n_+}$ denotes the regional fractional Laplacian 
\[
 \Dels{s}_{\R^n_+} u [\varphi] = \int_{\R^n_+} \int_{\R^n_+} \frac{(u(x)-u(y)) (\varphi(x)-\varphi(y))}{|x-y|^{n+2s}}\, dx\, dy.
\]

A few remarks are in order: 

Firstly, we observe that $u \in W^{s,2}(\R^n)$ implies that $u  \in L^p(\R^n_+)$, $p \in (2,\frac{n2}{n-2s})$ -- so for $p \in (2,\frac{n2}{n-2s}]$ the above theorem is an \emph{a posteriori estimate}, for $p > \frac{n2}{n-2s}$, \Cref{th:tangential} is an \emph{a priori} estimates. Usually we could try to bootstrap to a posteriori estimate by iteratively using the Sobolev embedding $W^{t,p}(\R^n) \subset L^{\frac{np}{n-tp}}$, but we cannot do this in our setting since we only have a tangential control.

Secondly, observe that in general our argument treats the case when $u$ may not be H\"older continuous, there is no restriction on $p > \frac{n}{2s}$ as e.g. in \cite{Fall22}. Our motivation to study the regularity theory comes from bootstrapping in critical nonlinear equations (e.g. from geometry or topology) where one needs to bootstrap to reach an assumption that the right-hand side belongs to $L^p$ for $p > \frac{n}{2s}$. However, since our argument relies on quantitative estimates of H\"older continuity for bounded right-hand sides, we need to restrict to the case $s > \frac{1}{2}$ where such H\"older continuity is available -- this is yet again somewhat unsatisfying, since it seems reasonable to that some tangential estimates might be true also when $s < \frac{1}{2}$.

Thirdly, we make no effort to obtain sharp results in the sense of $\tilde{s} = 2s-t$ -- while philosophically more appealing, for most applications (e.g. bootstrapping) weaker estimates such as ours usually suffice.

Lastly, we stress that \Cref{th:tangential} is void if the dimension $n=1$. Indeed the main role of \Cref{th:tangential} is (similar to the local case) to \emph{reduce} boundary regularity to the corresponding one-dimensional question. In the local case, one-dimensional PDEs are ODEs, thus with a simple regularity theory. Since we are nonlocal, our one-dimensional PDE is still nonlocal, but with a right-hand side that is as good as the tangential part from \Cref{th:tangential} allows. Precisely we have

\begin{theorem}[Reduction to one-dimensional nonlocal PDE]\label{th:normal}
Assume $s \in (\frac{1}{2},1)$, $\tilde{s} \in (s,1)$ such that $2s-\tilde{s} > \max\{t,\frac{1}{2}\}$. Assume that for $p > 2$,
\[
 u \in W^{s,2}(\R^n_+) \cap W^{\tilde{s},p}_T(\R^n_+)
\]
solves the equation \eqref{eq:ourpde} for $g$ as in \Cref{th:tangential}. Take any $\eta \in C_c^\infty(\R^{n-1})$ and set 
\[
 v(x_n) := \int_{\R^{n-1}} u(z',x_n)\, \eta(z') dz' \quad x_n \in (0,\infty).
\]
Then $v \in W^{s,2}_0((0,\infty))$ solves the one-dimensional equation
\[
 \Dels{s}_{(0,\infty)} v = \tilde{H} \quad \text{in $(0,\infty)$},
\]
where $\tilde{H} \in \brac{W^{2s-\tilde{s},p'}_0((0,\infty))}^{\ast}$ with the estimate
\[
\begin{split}
 |\tilde{H}(\psi)| \aleq& \brac{\|G\|_{L^p(\R^n)} + [u]_{W_T^{\tilde{s},p}(\R^n_+)}} \|\eta\|_{W^{2s-\tilde{s},p'}(\R^{n-1})} \|\psi\|_{W^{2s-\tilde{s},p'}((0,\infty))}.
 \end{split}
\]
\end{theorem}

The outline of the remaining paper is as follows: in \Cref{s:preliminaries} we introduce notation and preliminary arguments. In \Cref{s:L2est} we prove $L^2$-estimates, which in combination with the known H\"older regularity estimates leads to maximal function estimates in \Cref{s:maxfctest}. This is in the spirit of the arguments in \cite{DK12}, however we use discrete differentiation, whose embedding theorems we discuss in \Cref{s:embedding}. We then use the Fefferman-Stein estimate on maximal functions to obtain the proof of \Cref{th:tangential} in \Cref{s:tangential}. In \Cref{s:normal} we discuss the argument for \Cref{th:normal}.

Lastly, let us remark that our arguments could relatively easily extended to more general domains and more general differential operators, e.g. possibly those discussed in \cite{MSY21}. Since even in the simple case our results are limited, we prefer to favor a simpler notation over generality of domains.

\subsection*{Acknowledgment} Discussions with T. Mengesha, M.M. Fall, M. Warma are gratefully acknowledged. 
A substantial part of the research was carried out while A.S. was visiting Chulalongkorn University and Thammasat University, and their respective hospitality is appreciated.
A.S. is funded by Simons foundation grant no 579261 and NSF Career DMS-2044898. 
This paper was supported by Thammasat Postdoctoral Fellowship.

\section{Preliminaries and definitions}\label{s:preliminaries}
Throughout this paper we always assume $n \geq 2$ and $s \in (\frac{1}{2},1)$.
We denote by $\R^n_+$ the upper halfspace, and $\R^n_-$ the lower halfspace, i.e.
\[
 \R^n_+ := \{x = (x_1,\ldots,x_n) \in \R^n: \quad x_n > 0\}
 \]
 and
 \[ \R^n_- := \{x = (x_1,\ldots,x_n) \in \R^n: \quad x_n < 0\}
\]
and for any $\Omega \subset \R^n$, we set $\Omega_+ := \Omega \cap \R^n_+$ and $\Omega_{-} := \Omega \cap \R^{n}_-$.

For two open sets $\Omega, \Omega_1 \subset \R^n$ (typically, $\Omega_1 \subset \Omega$), we say
\[
\Dels{s}_{\Omega} u = f \quad \text{in $\Omega_1$},
\]
if for any $\varphi \in C_c^\infty(\Omega_1)$, we have
\[
\int_{\Omega} \int_{\Omega} \frac{\brac{u(x)-u(y)}\brac{\varphi(x)-\varphi(y)}}{|x-y|^{n+2s}}\, dx\, dy = f[\varphi].
\]

We define now the seminorm 
\[
 [f]_{W^{s,2}(\Omega)} := \brac{\int_{\Omega} \int_{\Omega} \frac{|f(x)-f(y)|^2}{|x-y|^{n+2s}}\, dx\, dy}^{\frac{1}{2}}
\]
and the norm
\[
 \|f\|_{W^{s,2}(\Omega)} := \|f\|_{L^2(\Omega)} + [f]_{W^{s,2}(\Omega)}.
\]
We say that $f \in \dot{W}^{s,2}(\Omega)$, or $W^{s,2}(\Omega)$ if and only $f$ is measurable and the above norm is finite. 
Moreover, we denote by
\[
W^{s,2}_0(\Omega) := \left \{v \in \dot{W}^{s,2}(\R^n) \cap L^{2}(\R^n), \quad v \equiv 0 \ \text{in $\R^n \sm \Omega$} \right \}
\]
and
\[
\mathcal{W}^{s,2}_0(\Omega) := \left \{v \in \dot{W}^{s,2}(\R^n) \cap L^{\frac{ns}{n-2s}}(\R^n), \quad v \equiv 0 \ \text{in $\R^n \sm \Omega$} \right \}.
\]

On $\R^n$ we define for $t >0$ the operators $\laps{t}$ and $\Ds{t}$ either by Fourier transform
\[
 \laps{t} f \equiv \Ds{t} f = \mathcal{F}^{-1} (c_1 |\xi|^t \mathcal{F} f),
\]
or, equivalently, via the integral representation. Namely, if $t \in (0,1)$
\[
 \laps{t} f(x) \equiv \Ds{t} f(x) = c_2 \int_{\R^n} \frac{f(y)-f(x)}{|x-y|^{n+t}}\, dy = c_2 \int_{\R^n} \frac{f(x+h)-f(x)}{|h|^{n+t}}\, dh,
\]
and if $t \in (0,2)$
\[
 \laps{t} f(x) \equiv \Ds{t} f(x) = \frac{c_2}{2}  \int_{\R^n} \frac{f(x+h)+f(x-h)-2f(x)}{|h|^{n+t}}\, dh.
\]
The inverse operator $\lapms{t} = \Ds{-t}$ can also written in potential form,
\[
 \lapms{t} f(x) \equiv \Ds{-t} f(x) = c_3 \int_{\R^n} |x-y|^{t-n}\, f(y)\, dy.
\]

\begin{lemma}\label{la:Ws0halfspace}
Let $s \in (\frac{1}{2},1)$ and $v \in W^{s,2}_0(\R^n_+)$. Then, 
\begin{equation}\label{eq:vws0est}
 [v ]_{W^{s,2}(\R^n)} \aleq [v]_{W^{s,2}(\R^n_+)} \quad \forall v \in W^{s,2}_0(\R^n_+).
\end{equation}
\end{lemma}
\begin{proof}
Since $v \equiv 0$ in $\R^{n}_-$, we have 
\[
 [v ]_{W^{s,2}(\R^n)} \aleq \left \|\frac{v}{\delta^s}\right \|_{L^2(\R^n_+)} + [v]_{W^{s,2}(\R^n_+)},
\]
where for $x = (x',x_n) \in \R^n \times (0,\infty)$ we have $\delta(x) = x_n$. Since $s > 1/2$, we use the fractional Hardy's inequality, \cite[Theorem 1.1.]{Dyda04}, on the first term in the equation above to conclude
\[
 [v ]_{W^{s,2}(\R^n)} \aleq [v]_{W^{s,2}(\R^n_+)}.\qedhere
\]
\end{proof}

We now define the usual maximal function
\[
 \mathcal{M} g(x) := \sup_{r >0} r^{-2n} \int_{B(x,r)} |g(y)| dy, \quad x \in \R^n
\]
and the sharp maximal function
\[
 \mathcal{M}^\# g(x) := \sup_{r >0} r^{-2n} \int_{B(x,r)} \, \int_{B(x,r)} |g(y)-g(z)|\, dy\, dz, \quad x \in \R^n
\]
as well as the censored versions
\[
 \mathcal{M}_+ g(x) := \sup_{r >0} r^{-2n} \int_{B(x,r) \cap \R^n_+} |g(y)| dy, \quad x \in \R^n_+
\]
and
\[
 \mathcal{M}^\#_{{+}} g(x) := \sup_{r >0} r^{-2n} \int_{B(x,r) \cap \R^n_+} \, \int_{B(x,r) \cap \R^n_+} |g(y)-g(z)|\, dy\, dz \quad x \in \R^n_+.
\]
The Fefferman-Stein theorem \cite{FS72} says that if $f \in L^p(\R^n)$ then 
\begin{equation}\label{eq:globFS}
 \|f\|_{L^p(\R^n)} \leq C(p) \|\mathcal{M}^\# f\|_{L^p(\R^n)}.
 \end{equation}
We remark that this also holds for the censored maximal functions.
\begin{theorem}[Censored Fefferman-Stein theorem]\label{th:feffermanstein}
If $f \in L^p(\R^n_+)$, then for a constant $C = C(p,n)$
\[
 \|f\|_{L^p(\R^n_+)} \leq C\, \|\mathcal{M}^\#_+ f\|_{L^p(\R^n_+)}.
 \]
\end{theorem}

\Cref{th:feffermanstein} is a direct consequence of the usual Fefferman-Stein estimate \eqref{eq:globFS} combined with the following observation.
\begin{lemma}\label{la:FefSteinhalf}
For measurable $f: \R^n_+ \to \R$ and we denote the even reflection to $\R^n$ by
\[
 \tilde{f}(x',x_n) := f(x',|x_n|).
\]
Then, for any $p \in [1,\infty)$,
\[
 \|\mathcal{M}^\# {\tilde{f}}\|_{L^p(\R^n)} \aleq
 \|\mathcal{M}^\#_{{+}} f\|_{L^p(\R^n_{{+}})}.
\]
\end{lemma}
\begin{proof}
For $x = (x_1,\ldots,x_{n-1},x_n)$, we denote the even reflection across $\R^{n-1} \times \{0\}$ by
\[
 x^\ast := (x_1,\ldots,x_{n-1},-x_n).
\]
Similarly, we denote the reflect set by
\[
 A^\ast = \{x^\ast:  x\in A\}.
\]
Let $x \in \R^n_+$ and $r > 0$. Then we consider
\[
\begin{split}
 &r^{-2n} \int_{B(x,r)} \, \int_{B(x,r)} |\tilde{f}(y)-\tilde{f}(z)|\, dy\, dz\\
=&r^{-2n} \int_{B(x,r)^+} \, \int_{B(x,r)^+} |f(y)-f(z)|\, dy\, dz\\
&+r^{-2n} \int_{B(x,r)^-} \, \int_{B(x,r)^-} |f(y^\ast)-f(z^\ast)|\, dy\, dz\\
&+2r^{-2n} \int_{B(x,r)^+} \, \int_{B(x,r)^-} |f(y^\ast)-f(z)|\, dy\, dz\\
=&r^{-2n} \int_{B(x,r)^+} \, \int_{B(x,r)^+} |f(y)-f(z)|\, dy\, dz\\
&+r^{-2n} \int_{\brac{B(x,r)^-}^\ast} \, \int_{\brac{B(x,r)^-}^\ast} |f(y)-f(z)|\, dy\, dz\\
&+2r^{-2n} \int_{B(x,r)^+} \, \int_{\brac{B(x,r)^-}^\ast} |f(y)-f(z)|\, dy\, dz.\\
 \end{split}
 \]
Now we use that
\[
 \brac{B(x,r)^-}^\ast \subset B(x,r)^+ \quad \text{if $x \in \R^n_+$}.
\]
Then we have shown
\[
\begin{split}
 &r^{-2n} \int_{B(x,r)} \, \int_{B(x,r)} |\tilde{f}(y)-\tilde{f}(z)|\, dy\, dz\\
\leq&4r^{-2n} \int_{B(x,r)^+} \, \int_{B(x,r)^+} |f(y)-f(z)|\, dy\, dz.\\
\end{split}
 \]
Thus,
\begin{equation}\label{eq:maxsharp1}
 \mathcal{M}^\# \tilde{f}(x) \aleq \mathcal{M}^\#_+ f(x) \quad \text{for all $x \in \R^n_+$}.
\end{equation}
We assume now $x \in \R^n_-$ and $r > 0$. Using that $B(x,r)^+ \subset (B(x,r)^-)^\ast$, we find 
\[
\begin{split}
 &r^{-2n} \int_{B(x,r)} \, \int_{B(x,r)} |\tilde{f}(y)-\tilde{f}(z)|\, dy\, dz\\
\leq&4r^{-2n} \int_{\brac{B(x,r)^-}^\ast} \, \int_{\brac{B(x,r)^-}^\ast} |{f}(y)-{f}(z)|\, dy\, dz.
 \end{split}
 \]
Moreover, we notice that
\[
\brac{B(x,r)^-}^\ast = B(x^\ast,r)^+.
\]
Thus, 
\[
\begin{split}
 &r^{-2n} \int_{B(x,r)} \, \int_{B(x,r)} |\tilde{f}(y)-\tilde{f}(z)|\, dy\, dz\\
 \leq& 4r^{-2n} \int_{B(x^\ast,r)^+} \, \int_{B(x^\ast,r)^+} |{f}(y)-{f}(z)|\, dy\, dz.\\
 \end{split}
 \]
That is, 
\begin{equation}\label{eq:maxsharp2}
 \mathcal{M}^\# \tilde{f}(x) \aleq \brac{\mathcal{M}^\#_+ f}(x^\ast) \quad \text{for all $x \in \R^n_-$}.
\end{equation}
Using \eqref{eq:maxsharp1} and \eqref{eq:maxsharp2} we have
\[
\begin{split}
 \|\mathcal{M}^\# \tilde{f}\|_{L^p(\R^n)} \aleq&\|\mathcal{M}^\# \tilde{f}\|_{L^p(\R^n_+)}+\|\mathcal{M}^\# \tilde{f}\|_{L^p(\R^n_-)}\\
 \aleq&\|\mathcal{M}^\# f\|_{L^p(\R^n_+)}+\|(\mathcal{M}^\# f)(\cdot^\ast)\|_{L^p(\R^n_-)}\\
 =&\|\mathcal{M}^\# f\|_{L^p(\R^n_+)}+\|(\mathcal{M}^\# f)(\cdot)\|_{L^p(\R^n_+)}.
 \end{split}
\]
We can conclude.
\end{proof}

\subsection{Discrete differentiation and embedding estimates}\label{s:embedding}
For $h \in \R^n$, we denote the discrete differentiation operators
\[
 \delta_h f(x) := f(x+h) - f(x)
\]
and 
\[
 \delta_{2,h} f(x) := f(x+h)+f(x-h) - 2f(x).
\]
We record the following well-known embedding estimates.
\begin{lemma}\label{la:embedding1}
Let $h \in \R^{n-1} \times \{0\}$, then
\[
 \|\delta_h f\|_{L^2(\R^n_+)} \aleq |h|^s [f]_{W^{s,2}_+(\R^n_+)}
\]
and if $|h| \leq 1$ and $R > 0$ then
\[
 \|\delta_h f\|_{L^2(B(0,R)^+)} \aleq |h|^s [f]_{W^{s,2}_+(B(0,R+1)^+)}.
\]

\end{lemma}
\begin{proof}
If $h = 0$, there is nothing to show as both sides are zero. 
Otherwise we observe that w.l.o.g. $|h| =1$. Indeed, assume
\[
 \|\delta_h f\|_{L^2(\R^n_+)} \aleq  [f]_{W^{s,2}_+(\R^n_+)}
\]
is proven for all $f$ and all $h \in \R^{n-1} \times \{0\}$ and $|h| = 1$.

Consider fix $f$ and $h \in \R^{n-1}\times \{0\} \sm \{0\}$. Set $g(x) := f(|h|x)$. From the above we assume we already have
 \[
 \|\delta_{\frac{h}{|h|}} g\|_{L^2(\R^n_+)} \aleq  [g]_{W^{s,2}_+(\R^n_+)},
\]
and thus
 \[
 \|f(|h|\cdot +h)- f(|h|\cdot)\|_{L^2(\R^n_+)} \aleq  |h|^{s-\frac{n}{2}} [f]_{W^{s,2}_+(\R^n_+)},
\]
and thus
 \[
 |h|^{-\frac{n}{2}} \|f(\cdot +h)- f(\cdot)\|_{L^2(\R^n_+)} \aleq  |h|^{s-\frac{n}{2}} [f]_{W^{s,2}_+(\R^n_+)},
\]
which implies the full claim.

So assume for now that $h \in \R^{n-1} \times \{0\}$ and $|h| = 1$ and we want to show
\[
 \|\delta_h f\|_{L^2(\R^n_+)} \aleq  [f]_{W^{s,2}_+(\R^n_+)}.
\]

Now observe that if $x \in B(x_0,1)^+$ then $x+h \subset B(x_0,2)^+$ (here we use that $h$ is tangential). We then have
\[
\begin{split}
 \|\delta_h f\|_{L^2(B(x_0,1)^+)} \leq& \|f(\cdot+h)-(f)_{B(x_0,2)^+}\|_{L^2(B(x_0,1)^+)}+\|f(\cdot)-(f)_{B(x_0,2)^+}\|_{L^2(B(x_0,1)^+)}\\
 \leq&2\|f(\cdot)-(f)_{B(x_0,2)^+}\|_{L^2(B(x_0,{2})^+)}\\
 \aleq&2 \brac{\int_{B(x_0,2)^+}\int_{B(x_0,2)^+} \frac{|f(x)-f(y)|^2}{|x-y|^{n+2s}}\, dx\, dy}^{\frac{1}{2}}\\
 \end{split}
\]
Squaring both sides we have
\[
\begin{split}
 \|\delta_h f\|_{L^2(B(x_0,1)^+)}^2 \aleq&\int_{B(x_0,2)^+}\int_{B(x_0,2)^+} \frac{|f(x)-f(y)|^2}{|x-y|^{n+2s}}\, dx\, dy\\
 \end{split}
\]
This holds for any $x_0$. We can cover $\R^n$ by countably many $\{B(x_k,1)\}_{k \in \N}$ so that each point $z \in \R^n$ is covered by a finite number $N \approx 2^n$ many balls of $\{B(x_k,2)\}_{k \in \N}$. Then we sum up the above inequality and have
\[
\begin{split}
\|\delta_h f\|_{L^2(\R^n_+)}^2 & \aleq \sum_{k \in \N} \|\delta_h f\|_{L^2(B(x_k,1)^+)}^2\\
\aleq&\sum_{k \in \N} \int_{B(x_k,2)^+}\int_{B(x_k,2)^+} \frac{|f(x)-f(y)|^2}{|x-y|^{n+2s}}\, dx\, dy\\
\leq&\sum_{k \in \N} \int_{B(x_k,2)^+}\int_{\R^n_+} \frac{|f(x)-f(y)|^2}{|x-y|^{n+2s}}\, dx\, dy\\
\aleq&N\int_{\R^n_+}\int_{\R^n_+} \frac{|f(x)-f(y)|^2}{|x-y|^{n+2s}}\, dx\, dy\\
 \end{split}
\]
This proves the claim.
\end{proof}

\begin{lemma}\label{la:del2todel1}
Let $0 < t_1 < 1$ and $t_2 \geq t_1$. Then
\[
 \sup_{h \in \R^{n-1} \times \{0\}} |h|^{-t_1} \|\delta_{h} f\|_{L^p(\R^n_+)}  \aleq \sup_{h \in \R^{n-1} \times \{0\}} |h|^{-t_2}\|\delta_{2,h} f\|_{L^p(\R^n_+)} + \|f\|_{L^p(\R^n_+)}
\]
\end{lemma}
\begin{proof}
The estimate for $t_2 > t_1$ follows readily from the case $t_2=t_1$. So in the following we assume $t_2 = t_1 = t \in (0,1)$.
Let $f \in L^p(\R^n_+)$ such that 
\[
\Lambda := \sup_{h \in \R^{n-1} \times \{0\}} |h|^{-t}\|\delta_{2,h} f\|_{L^p(\R^n_+)} < \infty.
\]
We use a trick from \cite[Lemma 2.3]{BL17}:
\[
 f(x)-f(x-h) = \frac{1}{2} \Big (\brac{f(x+h)-f(x-h)}-\brac{f(x+h)+f(x-h)-2f(x)} \Big ).
\]
That is
\[
 - \delta_{-h} f(x) =\frac{1}{2} \brac{\delta_{2h}f(x-h)-\delta_{2,h}f(x)}.
\]
Taking the $L^p$-norm we find
\[
 \|\delta_{h} f\|_{L^p(\R^n_+)} = \|\delta_{-h} f\|_{L^p(\R^n_+)} \leq\frac{1}{2} \|\delta_{2h}f\|_{L^p(\R^n_+)} + \frac{1}{2} \|\delta_{2,h}f\|_{L^p(\R^n_+)},
\]
and thus
\[
 |h|^{-t}\|\delta_{h} f\|_{L^p(\R^n_+)} \leq\frac{1}{2} 2^{t}|2h|^{-t}\|\delta_{2h}f\|_{L^p(\R^n_+)} + \frac{1}{2} |h|^{-t}\|\delta_{2,h}f\|_{L^p(\R^n_+)},
\]
For $\lambda > 0$ we set 
\[
 F(\lambda) := \sup_{h \in \R^{n-1} \times \{0\}:\, |h| = \lambda} |h|^{-t}\|\delta_{h} f\|_{L^p(\R^n_+)}.
\]
Clearly,
\[
 F(\lambda) \leq \lambda^{-t} 2 \|f\|_{L^p(\R^n_+)} < \infty \quad \text{for all }\lambda >0.
\]
In particular we find $F(+\infty) = 0$. The above inequality then implies for $\theta := 2^{t-1} \in (0,1)$
\[
 F(\lambda) \leq \theta F(2\lambda) + \frac{\Lambda}{2} \quad \forall \lambda > 0.
\]
Iterating this inequality we have for any $k \in \N$ and any $\lambda > 0$
\[
 F(\lambda) \leq \theta^k F(2^k \lambda) + \sum_{\ell = 0}^{k-1} \theta^\ell \frac{\Lambda}{2}.
\]
Thus, as $k \to \infty$ we have 
\[
 F(\lambda) \leq \frac{1}{1-\theta} \frac{\Lambda}{2}.
\]
This holds for any $\lambda > 0$, so we have 
\[
 \sup_{\lambda > 0} F(\lambda) \leq  \frac{1}{1-\theta} \frac{\Lambda}{2}.
\]
We conclude that 
\[
\sup_{h \in \R^{n-1} \times \{0\} \setminus \{0\}} |h|^{-t}\|\delta_{h} f\|_{L^p(\R^n_+)} \leq \frac{1}{2-2^{t}} \sup_{h \in \R^{n-1} \times \{0\}} |h|^{-t}\|\delta_{2,h} f\|_{L^p(\R^n_+)}. \qedhere
\]
\end{proof}

\begin{lemma}\label{la:tangentialest}
Let $t < t_2$, then
\[
 \brac{\int_{\R^n_+} \int_{\R^n_+} \frac{|u(x',{x_n})-u(y',{x_n})|^p}{|x-y|^{n+tp}} dx' dx_n dy' dy_n}^{\frac{1}{p}} \aleq \sup_{h \in \R^{n-1} \times \{0\}}|h|^{-t_2}\|\delta_{h}u\|_{L^p(\R^n_+)} + \|u\|_{L^p(\R^n_+)}.
\]
\end{lemma}
\begin{proof}
We have
\[
\begin{split}
 &\int_{y_n > 0} \frac{1}{|x-y|^{n+tp}} dy_n\\
 =&\int_{y_n > 0} \frac{1}{\brac{|x'-y'|^2 + |x_n-y_n|^2}^\frac{n+tp}{2}} dy_n\\
 \leq&\int_{z \in \R} \frac{1}{\brac{|x'-y'|^2 + |z|^2}^\frac{n+tp}{2}} dz\\
 =&|x'-y'|^{-n-tp} \int_{z \in \R} \frac{1}{\brac{1+ \abs{\frac{z}{|x'-y'|}}^2}^\frac{n+tp}{2}} dz\\
 =&|x'-y'|^{-n-tp} |x'-y'| \underbrace{\int_{\tilde{z} \in \R} \frac{1}{\brac{1+ \abs{\tilde{z}}^2}^\frac{n+tp}{2}} d\tilde{z}}_{=C(n,t,p) < \infty}\\
 =&C(n,t,p) |x'-y'|^{-(n-1)-tp}.
 \end{split}
\]
So
\[
\begin{split}
 &\int_{\R^n_+} \int_{\R^n_+} \frac{|u(x',x_n)-u(y',x_n)|^p}{|x-y|^{n+tp}} dx' dx_n dy' dy_n\\
 \leq&C\int_{0}^\infty \int_{\R^{n-1}} \int_{\R^{n-1}} \frac{|u(x',x_n)-u(y',x_n)|^p}{|x'-y'|^{n-1+tp}} dx'\,dy'\, dx_n  \\
  =&C\int_{h' \in \R^{n-1}}\int_{0}^\infty \int_{\R^{n-1}}  \frac{|u(x',x_n)-u(x'+h',x_n)|^p}{|h'|^{n-1+tp}} dx'\, dx_n\, dh' \\
 \end{split}
\]
So if we set $h = (h',0)$, then we have, setting $\Lambda := \sup_{h \in \R^{n-1} \times \{0\}}|h|^{-t_2}\|\delta_{h}u\|_{L^p(\R^n_+)}$,
\[
\begin{split}
 &\int_{\R^n_+} \int_{\R^n_+} \frac{|u(x',x_n)-u(y',x_n)|^p}{|x-y|^{n+tp}} dx' dx_n dy' dy_n\\
 \leq&C\int_{h' \in \R^{n-1}} |h|^{-(n-1+tp)}\int_{0}^\infty \int_{\R^{n-1}}  |u(x)-u(x+h)|^p dx'\, dx_n\, dh' \\
 =&C\int_{h' \in \R^{n-1}} |h|^{-(n-1+tp)} \|\delta_{h}u\|_{L^p(\R^n_+)}^p  dh' \\
 \aleq&\int_{h' \in \R^{n-1}, |h'| \leq 1} |h|^{-(n-1+tp)} \|\delta_{h}u\|_{L^p(\R^n_+)}^p  dh' +\int_{h' \in \R^{n-1}, |h'| \geq 1} |h|^{-(n-1+tp)} \|\delta_{h}u\|_{L^p(\R^n_+)}^p  dh'\\
 \aleq&\int_{h' \in \R^{n-1}, |h'| \leq 1} |h|^{-(n-1+(t-t_2)p)} dh' \Lambda^p +\|u\|_{L^p(\R^n_+)}^p  \underbrace{\int_{h' \in \R^{n-1}, |h'| \geq 1} |h|^{-(n-1+tp)} dh'}_{= C(t,p) < \infty}
 \end{split}
\]
We observe that since $t_2 > t$ we have $|h|^{-(n-1+(t-t_2)p)} = |h|^{(t_2-t)p-(n-1)} \in L^1(B(0,1) \cap \R^{n-1})$, so we have shown the claim.
\end{proof}

\section{Existence and local \texorpdfstring{$L^2$}{L2}-estimates}\label{s:L2est}
\begin{lemma}\label{la:localsol}
Let $s \in (\frac{1}{2},1)$, $n \geq 2$, $t \in [0,s)$ and $p \in [2,\infty)$. 
Then for any $G \in L^p(\R^n)$, $\lambda > 0$ and $x_0 \in \R^n_+$, there exists $w \in W^{s,2}(\R^n)$, $w \equiv 0$ in $\R^n_-$ and $\supp w \subset B(x_0,10\lambda)$ such that
\[
 \Dels{s}_{\R^n_+} w = \Ds{t} G \quad \text{in $B(x_0,\lambda)^+$},
\]
and we have the estimate
\begin{equation}\label{eq:la:localsolest}
 \lambda^{-\frac{n}{2}}\|w\|_{L^2(\R^n_+)} + \lambda^{s-\frac{n}{2}}[w]_{W^{s,2}(\R^n_+)} \aleq  \lambda^{2s-t} \brac{\mathcal{M}|G|^2(x_0)}^{\frac{1}{2}}.
\end{equation}

\end{lemma}
\begin{proof}
\underline{First assume $\lambda = 1$.}
Denote 
\[
 Y := \{v \in W^{s,2}(\R^n) :   \supp v \subset \overline{B(x_0,10)} \mbox{ and } v \equiv 0 \text{ in $\R^n_-$}\}
\]
and set 
\[
 \mathcal{E}(v) := \frac{1}{2}\int_{\R^n_+}\int_{\R^n_+} \frac{|v(x)-v(y)|^2}{|x-y|^{n+2s}}\, dx\, dy +\int_{\R^n} G \Ds{t}v.
\]
We observe that 
\[
 \abs{\int G \Ds{t}v} \aleq \|G\|_{L^p(\R^n)}\, \|\Ds{t} v\|_{L^{p'}(\R^n)}.
\]
Since $p' \leq 2$, $s \in (\frac{1}{2},1)$, $t <s$ and $n \geq 2$ we can use embedding theorems and the compactness of $\supp v$ to obtain
\[
\begin{split}
 \|\Ds{t} v\|_{L^{p'}(\R^n)} &\aleq \|v\|_{L^1(\R^n)} + [v]_{W^{s,2}(\R^n)} \\
 &\aleq C \|v\|_{L^{\frac{ns}{n-2s}}(\R^n)}+ [v]_{W^{s,2}(\R^n)} \\
 &\aleq C [v]_{W^{s,2}(\R^n)}.
 \end{split}
\]
Observe that constants are independent of $x_0$. By \Cref{la:Ws0halfspace}, we then have 
\[
 \|\Ds{t} v\|_{L^{p'}(\R^n)} \aleq [v]_{W^{s,2}(\R^n_+)}.
\]
By the same argument, we can also show 
\[
 \|\Ds{t} v\|_{L^{p'}(\R^n)} \aleq [v]_{W^{\bar{s},2}(\R^n_+)},
\]
for any $\bar{s} \in (\frac{1}{2},s)$, $\bar{s} > t$. In particular we find that 
\[
 v \mapsto \int_{\R^n} G \Ds{t}v
\]
is continuous with respect to the weak $W^{s,2}_0(\R^n_+)$-convergence (by Rellich's theorem, using that $\supp v \subset B(x_0,10)$).

In particular, we find 
\[
 [v]_{W^{s,2}(\R^n_+)}^2 \aleq \mathcal{E}(v) + \|G\|_{L^p(\R^n)}^2,
\]
with a constant independent of $x_0$. By Poincar\'e inequality, since $\supp v \subset B(x_0,10) \cap \R^n_+$,
\[
 \|v\|_{L^{2}(\R^n_+)} \aleq [v]_{W^{s,2}(\R^n_+)},
\]
again with a constant independent of $x_0$, so that we find 
\[
 \|v\|_{W^{s,2}(\R^n_+)}^2 \aleq \mathcal{E}(v) + \|G\|_{L^p(\R^n)}^2.
\]
Thus $\mathcal{E}$ is coercive in $Y$, and by the above considerations it is lower semicontinuous w.r.t weak $W^{s,2}_0$-convergence. We then find a minimizer $w \in Y$ of $\mathcal{E}$ in $Y$ by the direct method of Calculus of Variations.

Since for any $\varphi \in C_c^\infty(B(x_0,10)^+)$ we have $w + t\varphi \in Y$, we find that $w$ must satisfy the Euler-Lagrange equation
\[
 \Dels{s}_{\R^n_+} w = \Ds{t} G \quad \text{in $B(x_0,10)^+$}.
\]
By density we may testing this equation with $w$ itself, and we have 
\[
\begin{split}
[w]_{W^{s,2}(\R^n_+)}^2 &= \int_{\R^n} G\, \Ds{t} w\\
 &\leq \|G\|_{L^2(B(x_0,100))}\, \|\Ds{t} w\|_{L^{2}(\R^n)} + \int_{\R^n \setminus B(x_0,100)} |G|\, |\Ds{t} w|.
\end{split}
\]
If $t=0$, the last term is zero. If $t > 0$, then for $x \in \R^n \setminus B(x_0,100)$, by the support of $w$
\[
\begin{split}
|\Ds{t} w(x)| &\leq \int_{B(x_0,10)} |w(y)|\, |x-y|^{-n-t}\, dy\\
 &\aleq \brac{1+|x_0-x|}^{-n-t}\, \|w\|_{L^2(\R^n)}\\
  &\aleq \brac{1+|x_0 - x|}^{-n-t} [w]_{W^{s,2}(\R^n_+)}.
\end{split}
\]
So we have 
\[
\begin{split}
 [w]_{W^{s,2}(\R^n_+)}^2 
 \aleq& \brac{\|G\|_{L^2(B(x_0,100))}+  \|(1+|x_0 - \cdot|)^{-n-t}\, G\|_{L^1(\R^n)}} [w]_{W^{s,2}(\R^n_+)}\\
 \aleq& \brac{\mathcal{M} |G|^2}^{\frac{1}{2}}(x_0)\, [w]_{W^{s,2}(\R^n_+)}.
 \end{split}
\]
This implies \eqref{eq:la:localsolest} for $\lambda = 1$. \medskip

\underline{Now let $\lambda > 0$}

Consider $\bar{G}(x) := \lambda^{2s-t}G(\lambda x)$.

By the previous argument we find $\bar{w} \in W^{s,2}(\R^n_+)$, $\supp \bar{w} \subset B(\frac{x_0}{\lambda},10)$
\[
 \Dels{s}_{\R^n_+} \bar{w} = \Ds{t} \bar{G} \quad \text{in $B(\frac{x_0}{\lambda},1)^+$},
\]
and we have the estimate
\[
 \|\bar{w}\|_{L^2(\R^n_+)} + [w]_{W^{s,2}(\R^n_+)} \aleq \brac{\mathcal{M}|\bar{G}|^2(\lambda x_0)}^{\frac{1}{2}} = \lambda^{2s-t}\brac{\mathcal{M}|G|^2(\lambda x_0)}^{\frac{1}{2}}
\]
Set 
\[
 w(x) := \bar{w}(x/\lambda),
\]
then we have found $w \in W^{s,2}(\R^n_+)$, $\supp w \subset B(x_0,10\lambda)$ solving
\[
 \Dels{s}_{\R^n_+} w = \lambda^{-2s} \Ds{t} \bar{G} = G \quad \text{in $B(x_0,\lambda)^+$},
\]
and we have 
\[
  \lambda^{-\frac{n}{2}}\|w\|_{L^2(\R^n_+)}  =  \|\bar{w}\|_{L^2(\R^n)}
\]
and 
\[
  \lambda^{s-\frac{n}{2}} [w]_{W^{s,2}(\R^n_+)}  =   [\bar{w}]_{W^{s,2}(\R^n_+)}.\qedhere
\]
\end{proof}

Recall that for $h \in \R^n$, we denote the discrete differentiation operators
\[
 \delta_h f(x) := f(x+h) - f(x)
\]
and 
\[
 \delta_{2,h} f(x) := f(x+h)+f(x-h) - 2f(x).
\]

\begin{lemma}\label{la:globalsol3}
Let $n \geq 2$, $s \in (\frac{1}{2},1)$ and $G \in L^p(\R^n)$, $p \geq 2$, $t \in [0,s)$. 

Let $\lambda > 0$, $x_0 \in \R^n_+$ and take $\tilde{\eta} \in C_c^\infty(B(0,2))$, and set $\tilde{\eta}_{x_0,\lambda} := \tilde{\eta}((\cdot-x_0)/\lambda)$.

Then there exists $w \in W^{s,2}_0(\R^n_+)$, $\supp w \subset B(x_0,5\lambda)$, such that for any $h \in \R^{n-1} \times \{0\}$, $|h| < \lambda$,
\[
 \Dels{s}_{\R^n_+} \delta_{2,h} w = \delta_{2,h} \brac{\tilde{\eta}_{x_0,\lambda} \Ds{t} G} +g\quad \text{in $B(x_0,3\lambda) \cap \R^n_+$}.
\]
Moreover we have
\begin{equation}\label{eq:sol3:goalest1}
 \|g\|_{L^\infty(B(x_0,2\lambda))} \aleq_{\tilde{\eta}} |h|^{2s-t}\, \brac{\mathcal{M} |G|^2(x_0)}^{\frac{1}{2}}
\end{equation}
and 
\begin{equation}\label{eq:sol3:goalest2}
 \|\delta_{2,h} w\|_{L^2(\R^n)} \aleq \lambda^{\frac{n}{2}}\, |h|^{2s-t}\, \brac{\mathcal{M} |G|^2(x_0)}^{\frac{1}{2}}.
\end{equation}
\end{lemma}
\begin{proof}

We first assume that \underline{$\lambda =1$ and $x_0 \in \R^n_+$}.

Take $\tilde{w} \in W^{s,2}_0(\R^n_+)$, $\supp \tilde{w} \subset B(x_0,10)^+$ the solution of 
\[
  \Dels{s}_{\R^n_+} \tilde{w} = \tilde{\eta}_{x_0,1} \Ds{t} G \quad \text{in } B(x_0,10)^+.
\]
This exists as minimizer of the energy 
\[
 \mathcal{E}(v) := \frac{1}{2}[v]_{W^{s,2}(\R^n_+)}^2 + \int G \Ds{t} \brac{\tilde{\eta}_{x_0,1} \chi_{\R^n_+} v}.
\]
in the class 
\[
 V := \left \{v \in W^{s,2}_0(\R^n_+), \supp v \subset \overline{B(x_0,10)^+} \right \}.
\]
Observe that for $v \in V$ we have by Poincar\'e inequality (since we have compact support and $p' \leq 2$ -- observe the constant is independent of $x_0$), and using \eqref{eq:vws0est},
\[
 \|\Ds{t} \brac{\tilde{\eta}_{x_0,1} \chi_{\R^n_+} v}\|_{L^{p'}(\R^n_+)} \aleq_{\tilde{\eta}_{x_0,1}} \|v\|_{L^1(\R^n_+)}+ [\chi_{\R^n_+}v]_{W^{s,2}(\R^n)} \aleq [v]_{W^{s,2}(\R^n_+)}.
\]
Thus $\mathcal{E}$ is coercive and a minimizer $\tilde{w}$ exists, which satisfies (if $t= 0$ the second term is constantly zero)
\[
\begin{split}
[\tilde{w}]_{W^{s,2}(\R^n_+)}^2 \aleq& \abs{\int_{\R^n} G \Ds{t} \brac{\tilde{\eta}_{x_0,1} \chi_{\R^n_+} \tilde{w}}}\\
\aleq &\abs{\int_{B(x_0,10)} G \Ds{t} \brac{\tilde{\eta}_{x_0,1} \chi_{\R^n_+} \tilde{w}}} + \abs{\int_{\R^n \sm B(x_0,10)} G \Ds{t} \brac{\tilde{\eta}_{x_0,1} \chi_{\R^n_+} \tilde{w}}}\\
\aleq & \|G\|_{L^2(B(x_0,10))} [v]_{W^{s,2}(\R^n_+)} + \|\tilde{\eta}_{x_0,1}\|_{L^\infty}\int_{\R^n \sm B(x_0,10)} |G(x)| \int_{B(x_0,2)_+} |\tilde{w}(y)| |x-y|^{-n-t} dy\, dx\\
\aleq & \|G\|_{L^2(B(x_0,10))} [v]_{W^{s,2}(\R^n_+)} + \|\tilde{\eta}_{x_0,1}\|_{L^\infty}\int_{\R^n \sm B(x_0,10)} |G(x)| |1+|x_0-x||^{-n-t}\,  dx\, \|\tilde{w}\|_{L^1(\R^n)}\\
\aleq &\brac{\|G\|_{L^2(B(x_0,10))} + \int_{\R^n \sm B(x_0,10)} |G(x)| |1+|x_0-x||^{-n-t}\,  dx} [\tilde{w}]_{W^{s,2}(\R^n_+)} \\
\aleq &\brac{\mathcal{M}\brac{|G|^2}(x_0)}^{\frac{1}{2}} [\tilde{w}]_{W^{s,2}(\R^n_+)}.
\end{split}
\]
That is, we have 
\begin{equation}\label{eq:sol3:2q3423r}
 \|\tilde{w}\|_{L^2(\R^n_+)} + [\tilde{w}]_{W^{s,2}(\R^n_+)} \aleq \brac{\mathcal{M}\brac{|G|^2}(x_0)}^{\frac{1}{2}}.
\end{equation}
The constant is independent of $x_0$.

Take $\eta \in C_c^\infty(B(x_0,5))$, $\eta \equiv 1$ in $B(x_0,4)$ (we can assume it is a $x_0$-translation of some generic $\eta$). We have
 \[
  \Dels{s}_{\R^n_+} (\eta \tilde{w}) = \eta \tilde{\eta}_{x_0,1} \Ds{t} G + [\Dels{s},\eta](\tilde{w}) \quad \text{in } B(x_0,10)^+.
 \]
Set $w := \eta \tilde{w}$, then $\supp w \subset B(x_0,5)$. We have
\begin{equation}\label{eq:ex2:Delsw}
  \Dels{s}_{\R^n_+} w = \eta \tilde{\eta}_{x_0,1} \Ds{t} G + [\Dels{s},\eta](\tilde{w}) \quad \text{in } B(x_0,10)^+.
 \end{equation}

From \eqref{eq:ex2:Delsw} we find that for any $|h| \leq 1$, $h \in \R^{n-1} \times \{0\}$,
 \[
  \Dels{s}_{\R^n_+} (\delta_{2,h}w) = \delta_{2,h} \brac{\eta \tilde{\eta}_{x_0,1} \Ds{t} G} + \delta_{2,h} \brac{[\Dels{s},\eta](\tilde{w})} \quad \text{in } B(x_0,9)^+.
 \]
Here, we use the commutator notation $[T,\eta](v) = T(\eta v) - \eta Tv$. 
Since $\tilde{\eta}_{x_0,1} \eta = \tilde{\eta}_{x_0,1}$ in $B(x_0,4)$ we have 
\[
  \Dels{s}_{\R^n_+} (\delta_{2,h}w) = \delta_{2,h} \brac{\tilde{\eta}_{x_0,1} \Ds{t} G} + g \quad \text{in } B(x_0,3)^+.
 \]
where we set
\[
 g = \delta_{2,h}[\Dels{s},\eta](\tilde{w}).
\]
If $x \in B(x_0,3)_+$, since $\eta \equiv 1$ in $B(x_0,4)$,
\[
\begin{split}
 &[\Dels{s},\eta](\tilde{w})(x) \\
 =& c\int_{\R^n} \frac{\tilde{w}(x+z) \eta(x+z) + \tilde{w}(x-z) \eta(x-z) - 2 \tilde{w}(x) \eta(x) - \eta(x) \brac{\tilde{w}(x+z) +\tilde{w}(x-z) -2\tilde{w}(x) }}{|z|^{n+{2s}}}\, dz\\
 =& c\int_{\R^n} \frac{\tilde{w}(x+z) (\eta(x+z) -\eta(x) )+ \tilde{w}(x-z) (\eta(x-z)  - \eta(x)) }{|z|^{n+{2s}}}\, dz\\
 =& c\int_{\R^n} \frac{\tilde{w}(x+z) (\eta(x+z) -1 )+ \tilde{w}(x-z) (\eta(x-z)  - 1) }{|z|^{n+{2s}}}\, dz\\
 =& 2c\int_{\R^n_+} \frac{\tilde{w}(y) (\eta(y) -1 )}{|x-y|^{n+{2s}}}\, dy= 2c\int_{\R^n_+ \sm B(x_0,4)} \frac{\tilde{w}(y) (\eta(y) -1 )}{|x-y|^{n+{2s}}}\, dy.\\
 \end{split}
\]
Thus, for $x \in B(x_0,2)$, $h \in \R^{n-1} \times \{0\}$ 
\[
 \abs{\delta_{2,h}[\Dels{s},\eta](\tilde{w})}(x) \aleq \int_{\R^n_+ \sm B(x_0,4)} |\tilde{w}(y)| \abs{|x+h-y|^{-n-2s} +|x-h-y|^{-n-2s}-2|x-y|^{-n-2s}}\, dy.
\]
By Taylor's theorem for $x \in B(x_0,2)$ and $y \in \R^n \sm B(x_0,4)$ and $|h| \leq 1$,
\[
\abs{|x+h-y|^{-n-2s} +|x-h-y|^{-n-2s}-2|x-y|^{-n-2s}} \aleq \sup_{|\tilde{h}| \leq 1} |x+\tilde{h}-y|^{-n-2s-2} |h|^2.
\]
Observe that for $x \in B(x_0,2)_+$, $y \in \R^n_+ \sm B(x_0,4)$, and $|\tilde{h}|\leq 1$ we have 
\[
 \sup_{|\tilde{h}| \leq 1} |x+\tilde{h}-y|^{-n-2s-2} \aleq \brac{|x-y|+1}^{-n-2s-2}.
\]
So we have for $x \in B(x_0,2)$, $h \in \R^{n-1} \times \{0\}$ 
\[
 \begin{split}
 \abs{\delta_{2,h}[\Dels{s},\eta](\tilde{w})}(x) \aleq& |h|^2 \int_{\R^n_+ \sm B(x_0,4)} |\tilde{w}(y)| \brac{|x-y|+1}^{-n-2s-2} dy\\
 \aleq & |h|^2 \|\tilde{w}\|_{L^2(\R^n)}.\\
 \end{split}
\]
This readily implies for any $h \in \R^{n-1} \times \{0\}$
\[
 \|g\|_{L^\infty(B(x_0,2))} \aleq |h|^2 \|\tilde{w}\|_{L^2(\R^n_+)} \overset{\eqref{eq:sol3:2q3423r}}{\aleq} |h|^2 \brac{\mathcal{M}\brac{|G|^2}(x_0)}^{\frac{1}{2}}.
\]
Thus \eqref{eq:sol3:goalest1} is proven for $\lambda = 1$.

As for \eqref{eq:sol3:goalest2}, since $\supp w \subset B(x_0,5)$ 
\[
\begin{split}
 \|\delta_{2,h} w\|_{L^2(\R^n)} =& \|\delta_{2,h} w\|_{L^2(B(x_0,6))}\\
  =&\|(\delta_{h} \delta_h w)(\cdot+h)\|_{L^2(B(x_0,6))}\\
  \leq&\|\delta_{h} \delta_h w\|_{L^2(B(x_0,7))}\\
  \overset{\text{L~\ref{la:embedding1}}}{\aleq}& |h|^s [\delta_h w]_{W^{s,2}(B(x_0,8))}.
 \end{split}
\]
Now we have from \eqref{eq:ex2:Delsw} for any $h \in \R^{n-1} \times \{0\}$, $|h| \leq 1$,
\[
   \Dels{s}_{\R^n_+} (\delta_{h}w) = \delta_{h} \brac{\eta \tilde{\eta}_{x_0,1} \Ds{t} G} + \delta_h [\Dels{s},\eta](\tilde{w}) \quad \text{in } B(x_0,9)^+.
  \]
Since $\supp w \subset B(x_0,5)$ we have that $\supp \delta_h w \subset B(x_0,6)$, so we can test this equation with $\delta_h w$. Then we get 
\[
\begin{split}
  [\delta_h w]_{W^{s,2}(\R^n_+)}^2 =& \int_{\R^n_+} \brac{\delta_{h} \brac{\eta \tilde{\eta}_{x_0,1} \Ds{t} G} + \delta_h [\Dels{s},\eta](\tilde{w})}\, \delta_h w\\
  =& \int_{\R^n_+} \brac{\delta_{h} \lapms{s}\brac{\eta \tilde{\eta}_{x_0,1} \Ds{t} G} + \delta_h \lapms{s} \brac{[\Dels{s},\eta](\tilde{w})}}\, \Ds{s}\delta_h w\\
 \aleq&\| \delta_h \lapms{s}\brac{\eta \tilde{\eta}_{x_0,1} \Ds{t} G}\|_{L^2(\R^n)}\, [\delta_h w]_{W^{s,2}(\R^n)} 
   + \|\delta_h \lapms{s}\brac{[\Dels{s},\eta](\tilde{w})}\|_{L^2(\R^n)} [\delta_h w]_{W^{s,2}(\R^n)}.
 \end{split}
\]
We can apply \Cref{la:Ws0halfspace} to $\delta_h w$, divide both sides by $[\delta_h w]_{W^{s,2}(\R^n)}$ and then apply \Cref{la:embedding1}, to obtain
\[
\begin{split}
  [\delta_h w]_{W^{s,2}(\R^n)}  \aleq&\|\lapms{s} \delta_h \brac{\eta \tilde{\eta}_{x_0,1} \Ds{t} G}\|_{L^2(\R^n)}
   + \|\lapms{s} \brac{\delta_h [\Dels{s},\eta](\tilde{w})}\|_{L^2(\R^n)}\\
   \aleq& |h|^{s-t} [\lapms{s} \brac{\eta \tilde{\eta}_{x_0,1} \Ds{t} G}]_{W^{s-t,2}(\R^n)} + |h|^s \| [\Dels{s},\eta](\tilde{w})\|_{L^2(\R^n)}\\
    \aleq&|h|^{s-t} \|\lapms{t} \brac{\eta \tilde{\eta}_{x_0,1} \Ds{t} G}\|_{L^2} + |h|^s \|[\Dels{s},\eta](\tilde{w})\|_{L^2(\R^n)}\\
    \aleq&|h|^{s-t} \brac{
    \|\lapms{t} \brac{\eta \tilde{\eta}_{x_0,1} \Ds{t} \brac{\chi_{B(x_0,20)} G}}\|_{L^2} + \|\lapms{t} \brac{\eta \tilde{\eta}_{x_0,1} \Ds{t} \brac{\chi_{\R^n \sm B(x_0,20)}G} }\|_{L^2}
    } + |h|^s \|[\Dels{s},\eta](\tilde{w})\|_{L^2(\R^n)}\\
    \aleq&|h|^{s-t} \brac{\brac{\mathcal{M} |G|^2}^{\frac{1}{2}}(x_0) + \|\eta \tilde{\eta}_{x_0,1} \Ds{t} \brac{\chi_{\R^n \sm B(x_0,20)}G}\|_{L^{\frac{n2}{n+t2}}(\R^n)}} + |h|^s [\tilde{w}]_{W^{s,2}(\R^n)}.
 \end{split}
\]
In the last lines we used repeatedly Sobolev embedding theorems and the arguments for the estimates from before, as well as the (non-sharp) commutator estimate 
\[
\begin{split}
 \|[\Dels{s},\eta](\tilde{w})\|_{L^2(\R^n)} \aleq& \|\eta\|_{C^2} \brac{\|\Ds{s} \tilde{w}\|_{L^2(\R^n)} + \|\tilde{w}\|_{L^2(\R^n)}}\\
  \aleq&\|\eta\|_{C^2} [\tilde{w}]_{W^{s,2}(\R^n)}.
 \end{split}
\]
which follows from the representation
\[
\begin{split}
 [\Dels{s},\eta](\tilde{w})(x) =& c \int_{\R^n} \frac{\brac{\eta(x+z)-\eta(x)} \brac{\tilde{\omega}(x+z)-\tilde{\omega}(x-z)}}{|z|^{n+2s}}\, dz\\
  &+c \int_{\R^n} \frac{\brac{\eta(x+z)+\eta(x-z)-2\eta(x)} \tilde{\omega}(x-z)}{|z|^{n+2s}}\, dz,
 \end{split}
\]
along the lines of the previous estimates of this type.

Again we observe that 
\[
\begin{split}
 &\abs{\eta \tilde{\eta}_{x_0,1} \Ds{t} \brac{\chi_{\R^n \sm B(x_0,20)}G}}(x)\\
 \aleq&\chi_{B(x_0,10)}(x) \int_{\R^n \sm B(x_0,20)} |G(y)| |x-y|^{-n-t}\, dy\\
 \aleq&\chi_{B(x_0,10)}(x) \int_{\R^n \sm B(x_0,20)} |G(y)| \brac{1+|x_0-y|}^{-n-t}\, dy\\
 \aleq&\chi_{B(x_0,10)}(x)  \mathcal{M} |G|(x_0).\\
 \end{split}
\]
Consequently,
\[
 \|\eta \tilde{\eta}_{x_0,1} \Ds{t} \brac{\chi_{\R^n \sm B(x_0,20)}G}\|_{L^{\frac{n2}{n+t2}}(\R^n)} \aleq \mathcal{M} |G|(x_0).
\]
We thus have shown (using \eqref{eq:sol3:2q3423r} and \Cref{la:Ws0halfspace}),
\[
\begin{split}
  [\delta_h w]_{W^{s,2}(\R^n)}  \aleq&|h|^{s-t}\, \brac{\mathcal{M} |G|^2}^{\frac{1}{2}}(x_0) + |h|^s [\tilde{w}]_{W^{s,2}(\R^n)}\\
  \aleq&|h|^{s-t}\, \brac{\mathcal{M} |G|^2}^{\frac{1}{2}}(x_0) + |h|^s [\tilde{w}]_{W^{s,2}(\R^n_+)}\\
  \aleq&|h|^{s-t}\, \brac{\mathcal{M} |G|^2}^{\frac{1}{2}}(x_0) + |h|^2 \brac{\mathcal{M} |G|^2}^{\frac{1}{2}}(x_0).\\
 \end{split}
\]
This implies \eqref{eq:sol3:goalest2} for $\lambda = 1$.\medskip

\underline{Now assume $\lambda > 0$ and $x_0 \in \R^n$}.

For $\bar{G}(x) := \lambda^{2s-t} G(\lambda x)$, $\bar{h} = \frac{h}{\lambda}$, equivalently
\[
 G(x) = \lambda^{t-2s} \bar{G}(x/\lambda),
\]
we find we find $\bar{g}$ and $\bar{w}$ such that for any $|\bar{h}| \leq 1$ and 
\[
 \Dels{s}_{\R^n_+} \delta_{2,\bar{h}} \bar{w} = \delta_{2,\bar{h}} \brac{\tilde{\eta}_{x_0/\lambda,1} \Ds{t} \lambda^{-2s} \bar{G}} +\bar{g}\quad \text{in $B(\frac{x_0}{\lambda},3) \cap \R^n_+$}
\]
Setting 
\[
 w(x) := \bar{w}(x/\lambda)
\]
and 
\[
 g(x) := \lambda^{-2s} \bar{g}(x/\lambda)
\]
we have
\[
\Dels{s}_{\R^n_+} \delta_{2,\lambda \bar{h}} w = \delta_{2,\lambda \bar{h}} \brac{\tilde{\eta}_{x_0,\lambda} \Ds{t} G} +g\quad \text{in $B(x_0,3\lambda) \cap \R^n_+$}.
\]
and the estimates 
\[
 \|g\|_{L^\infty(B(x_0,2\lambda))} = \lambda^{-2s} \|\bar{g}\|_{L^\infty(B(x_0/\lambda,2)} \aleq |\lambda \bar{h}|^{2s-t} \brac{\mathcal{M} |G|^2(x_0)}^{\frac{1}{2}}
\]
and 
\[
 \|\delta_{2,\lambda \bar{h}} w\|_{L^2(\R^n)} = \lambda^{\frac{n}{2}} \|\delta_{2,\bar{h}} \bar{w}\|_{L^2(\R^n)} \aleq \lambda^{\frac{n}{2}} |\lambda \bar{h}|^{2s-t} \brac{\mathcal{M} |G|^2(x_0)}
\]
Using that $\lambda \bar{h} = h$, we conclude.
\end{proof}

\section{Tangential Improvement: Estimate of maximal functions}\label{s:maxfctest}

We follow the spirit of the argument in Dong-Kim \cite{DK12} and begin with the already known H\"older estimate -- which is the main reason that we are restricted to $s > \frac{1}{2}$. As usual we denote by $(f)_{\Omega} := \mvint_{\Omega} f$ the mean value.
\begin{lemma}\label{la:sharpest2}
Assume $s > \frac{1}{2}$ and $n \geq 1$ then there exists $\sigma = \sigma(n,s)$ and a constant $C = C(n,s)$ such that for any $\kappa \geq 2$, $r > 0$, and $x_0 \in \R^n$ we have the following: 

Let $v \in W^{s,2}_0(\R^n_+)$ solve
\[
 \Dels{s}_{\R^n_+} v =h \quad \text{in $B(x_0,2\kappa r)^+$}.
\]
Then, for some $\alpha > 0$
\[
\begin{split}
   \brac{|v-(v)_{B(x_0,r)^+}|}_{B(r,x_0)^+} \leq& C(n,s) \kappa^{-\alpha} \brac{\sum_{k=0}^\infty 2^{-\sigma k} \mvint_{B(x_0,2^k \kappa r)} |v|^2 dx}^{\frac{1}{2}}\\
   &+ C(\kappa)\|h\|_{L^\infty(B(x_0,2\kappa r)\cap \R^n_+)}
   \end{split}
\]
\end{lemma}
\begin{proof}
If $x_0 = 0$, $\kappa = 2$, $r = 1$ we have from \cite[Lemma 6.1]{Fall22} 
\[
  [v]_{C^\alpha(B_1(0)^+)} \aleq \|v\|_{L^2(B_2(0)^+)} + \int_{\R^n_+} \frac{|v(y)|}{1+|y|^{d+2s}}\, dy + \|h\|_{L^\infty(B(x_0,2\kappa r)\cap \R^n_+)}
\]
which readily implies 
\[
  [v]_{C^\alpha(B_1(0)^+)} \aleq \brac{\sum_{k=0}^\infty 2^{-\sigma k} \mvint_{B(x_0,2^k)} |v|^2 dx}^{\frac{1}{2}}+ \|h\|_{L^\infty(B(x_0,2\kappa r)\cap \R^n_+)}
\]
The claim now follows from applying this estimate to 
\[
 \tilde{v}(x) := v(\kappa rx). \qedhere
\]
\end{proof}

From the H\"older estimate \Cref{la:sharpest2} and a local a priori estimates we can obtain an estimate of maximal functions, which is the main result in this section.
\begin{proposition}\label{pr:maxfctest}
Let $s \in (\frac{1}{2},1)$, $n \geq 2$. For any $\kappa \geq  2$ there exists $C(\kappa,s,n)$ such that if $u \in W^{s,2}_0(\R^n_+)$  solves
\[
 \Dels{s}_{\R^n_+} u = \Ds{t} G \quad \text{in $\R^n_+$}
\]
we have for any $h \in \R^{n-1} \times \{0\}$, any $x_0 \in \R^n_+$ and any $r > 0$
\begin{equation}\label{eq:280523:goal}
\begin{split}
 \mathcal{M}^\#_+\delta_{2,h}u(x_0) \aleq& \kappa^{-\alpha} \brac{\mathcal{M}_+|\delta_{2,h}u|^2}^{\frac{1}{2}}(x_0) \\
   & +C(\kappa) |h|^{2s-t}\, \brac{\brac{\mathcal{M} |G|^2(x_0)}^{\frac{1}{2}}+\brac{\mathcal{M} |G|^2(x_0+h)}^{\frac{1}{2}}+\brac{\mathcal{M} |G|^2(x_0-h)}^{\frac{1}{2}}}
 \end{split}
\end{equation}
The constants in $\aleq$ depend only on $n$, $s$, in particular they are independent of $f$, $h$, and $\kappa$.
\end{proposition}

\begin{proof}[Proof of \Cref{pr:maxfctest}]
Fix some $h$ which is tangential, i.e. $h \in \R^{n-1}\times \{0\}$. Then we have
\[
\begin{cases}
  \Dels{s}_{\R^n_+} \delta_{2,h} u =  \Ds{t} \delta_{2,h} G \quad \text{in $\R^n_+$}\\
  \delta_{2,h} u = 0 \quad \text{on $\R^{n-1}\times \{0\}$}
\end{cases}
\]
Pick now any $x_0 \in \R^n_+$, $\kappa > 2$ and some $r > 0$. We consider two cases:

\underline{Assume $|h| \geq \kappa r$}
In this case we use \Cref{la:localsol} for $\lambda = 2\kappa r$ to solve,
\[
 \Dels{s}_{\R^n_+} w = \Ds{t} \delta_{2,h}G \quad \text{in $B(x_0,2\kappa r)^+$}
\]
which comes with the estimate
\[
\begin{split}
 (\kappa r)^{-\frac{n}{2}} \|w\|_{L^2(\R^n_+)} \aleq&  (\kappa r)^{2s-t} \brac{\mathcal{M}|\delta_{2,h}G|^2(x_0)}^{\frac{1}{2}}\\
 \aleq&  |h|^{2s-t} \brac{\brac{\mathcal{M}|G|^2(x_0)}^{\frac{1}{2}}+\brac{\mathcal{M}|G|^2(x_0-h)}^{\frac{1}{2}}+\brac{\mathcal{M}|G|^2(x_0+h)}^{\frac{1}{2}}}\\
 \end{split}
\]
That is
\begin{equation}\label{eq:24857346}
 r^{-\frac{n}{2}} \|w\|_{L^2(\R^n_+)} \aleq C(\kappa) |h|^{2s-t} \brac{\brac{\mathcal{M}|G|^2(x_0)}^{\frac{1}{2}}+\brac{\mathcal{M}|G|^2(x_0-h)}^{\frac{1}{2}}+\brac{\mathcal{M}|G|^2(x_0+h)}^{\frac{1}{2}}}
 \end{equation}
Consequently, 
\[
 \Dels{s}_{\R^n_+} (\delta_{2,h} u-w) = 0 \quad \text{in $B(x_0,\kappa r)^+$}.
\]
We apply \Cref{la:sharpest2} to $v:=\delta_{2,h} u-w$, and have
\[
   \brac{|v-(v)_{B(r)^+}|}_{B(r,x_0)^+} \aleq \kappa^{-\alpha} \brac{\sum_{k=0}^\infty 2^{-\sigma k} \mvint_{B(x_0,2^k \kappa r)} |v|^2 dx}^{\frac{1}{2}}.
\]
By triangular inequality we find
\[
\begin{split}
 &\brac{|\delta_{2,h}u-(\delta_{2,h} u)_{B(r)^+}|}_{B(r,x_0)^+} \\
 \aleq& \brac{|w-(w)_{B(r)^+}|}_{B(r,x_0)^+} \\
 &+ \kappa^{-\alpha} \brac{\sum_{k=0}^\infty 2^{-\sigma k} \mvint_{B(x_0,2^k \kappa r)^+} |\delta_{2,h} u|^2 dx}^{\frac{1}{2}} + \kappa^{-\alpha} \brac{\sum_{k=0}^\infty 2^{-\sigma k} \mvint_{B(x_0,2^k \kappa r)^+} |w|^2 dx}^{\frac{1}{2}}\\
\aleq& r^{-\frac{n}{2}} \|w\|_{L^2(\R^n_+)} \\
 &+ \kappa^{-\alpha} \brac{\sum_{k=0}^\infty 2^{-\sigma k} \mvint_{B(x_0,2^k \kappa r)^+} |\delta_{2,h} u|^2 dx}^{\frac{1}{2}} + \kappa^{-\alpha} \brac{\sum_{k=0}^\infty 2^{-\sigma k} (2^{k} \kappa r)^{-n} \|w\|_{L^2(\R^n_+)}^2}^{\frac{1}{2}}\\
 \aleq& C(\kappa) r^{-\frac{n}{2}} \|w\|_{L^2(\R^n_+)} + \kappa^{-\alpha} \brac{\sum_{k=0}^\infty 2^{-\sigma k} \mvint_{B(x_0,2^k \kappa r)^+} |\delta_{2,h} u|^2 dx}^{\frac{1}{2}}\\
\overset{\eqref{eq:24857346}}{\aleq}&C(\kappa) |h|^{2s-t} \brac{\brac{\mathcal{M}|G|^2(x_0)}^{\frac{1}{2}} +\brac{\mathcal{M}|G|^2(x_0+h)}^{\frac{1}{2}}
+\brac{\mathcal{M}|G|^2(x_0-h)}^{\frac{1}{2}}
} \\
&+ \kappa^{-\alpha} \brac{\sum_{k=0}^\infty 2^{-\sigma k} \mvint_{B(x_0,2^k \kappa r)^+ } |\delta_{2,h} u|^2 dx}^{\frac{1}{2}}\\
 \end{split}
\]
That is, we have shown
\begin{equation}\label{eq:280523:largeh}
\begin{split}
& \sup_{\kappa r \leq |h|}\brac{|\delta_{2,h}u-(\delta_{2,h} u)_{B(r)^+}|}_{B(r,x_0)^+} \\
\aleq& \kappa^{-\alpha} \brac{\mathcal{M}|\delta_{2,h} u|^2(x_0)}^{\frac{1}{2}} \\
&+C(\kappa) |h|^{2s-t} \brac{\brac{\mathcal{M}|G|^2(x_0)}^{\frac{1}{2}} +\brac{\mathcal{M}|G|^2(x_0+h)}^{\frac{1}{2}}+\brac{\mathcal{M}|G|^2(x_0-h)}^{\frac{1}{2}}}\\
 \end{split}
\end{equation}
We point out the similarity to \eqref{eq:280523:goal}.

\underline{We now assume $|h| \leq \kappa r$}. We need to get the same estimate as \eqref{eq:280523:largeh}.

In this case let $\eta \in C_c^\infty(B(x_0,10 \kappa  r)$, $\eta \equiv 1$ in $B(x_0,9 \kappa  r)$.

By \Cref{la:globalsol3} for $\lambda =  5 \kappa  r$ we find $w \in \dot{W}^{s,2}_0(\R^n_+)$, $\supp w \subset B(x_0,25 \kappa  r)$ that solves
\[
 \Dels{s}_{\R^n_+} \delta_{2,h} w = \delta_{2,h} \brac{\eta \Ds{t} G} +g\quad \text{in $B(x_0,15 \kappa  r) \cap \R^n_+$}
\]
and with the estimates 
\[
 \|g\|_{L^\infty(B(x_0,20 \kappa  r))} \aleq |h|^{2s-t}\, \brac{\mathcal{M} |G|^2(x_0)}^{\frac{1}{2}},
\]
and 
\[
 \|\delta_{2,h} w\|_{L^2(\R^n)} \aleq ( \kappa  r)^{\frac{d}{2}}\, |h|^{2s-t}\, \brac{\mathcal{M} |G|^2(x_0)}^{\frac{1}{2}}.
\]
We then have
\[
 \Dels{s}_{\R^n_+} \brac{ \delta_{2,h}(u-w)} = \delta_{2,h} \brac{(1-\eta)\Ds{t} G} - g \quad \text{in $B(x_0,15 \kappa  r)$}
\]
Since $|h| \leq \kappa r$, and $\eta \equiv 1$ in $B(x_0,9\kappa r)$ we have that
\[
 \delta_{2,h} \brac{(1-\eta)\Ds{t} G} = 0 \text{in $B(x_0,8 \kappa r\cap \R^n_+$)}.
\]
So if we set
\[
 v := \brac{ \delta_{2,h}u-\delta_{2,h}w}
\]
then 
\[
 \Dels{s}_{\R^n_+} v = -g  \quad \text{in $B(x_0,2\kappa r) \cap \R^n_+$}.
\]
We can apply \Cref{la:sharpest2}
\[
\begin{split}
   &\brac{|\delta_{2,h}u-(\delta_{2,h}u)_{B(r)^+}|}_{B(r,x_0)^+}\\
   \aleq&
   \kappa^{-\alpha} \brac{\sum_{k=0}^\infty 2^{-\sigma k} \mvint_{B(x_0,2^k \kappa r)} |\delta_{2,h}u|^2 dx}^{\frac{1}{2}}\\
   &+ \brac{|\delta_{2,h}w-(\delta_{2,h}w)_{B(r)^+}|}_{B(r,x_0)^+} + \kappa^{-\alpha} \brac{\sum_{k=0}^\infty 2^{-\sigma k} \mvint_{B(x_0,2^k \kappa r)} |\delta_{2,h}w|^2 dx}^{\frac{1}{2}}\\
   &+C(\kappa)\|g\|_{L^\infty(B(x_0,2\kappa r))}\\
   \aleq&   \kappa^{-\alpha} \brac{\mathcal{M}_+|\delta_{2,h}u|^2}^{\frac{1}{2}} + C(\kappa) r^{-\frac{n}{2}} \|\delta_{2,h}w\|_{L^2(\R^n_+)}+C(\kappa)\|g\|_{L^\infty(B(x_0,2\kappa r))}\\
   \aleq&\kappa^{-\alpha} \brac{\mathcal{M}_+|\delta_{2,h}u|^2}^{\frac{1}{2}} 
   + C(\kappa) |h|^{2s-t}\, \brac{\mathcal{M} |G|^2(x_0)}^{\frac{1}{2}}
   \end{split}
\]
which holds for all $\kappa r \geq |h|$. Combining this with \eqref{eq:280523:largeh}, we have shown
\[
 \begin{split}
& \sup_{r >0}\brac{|\delta_{2,h}u-(\delta_{2,h} u)_{B(r)^+}|}_{B(r,x_0)^+} \\
\aleq& \kappa^{-\alpha} \brac{\mathcal{M}|\delta_{2,h} u|^2(x_0)}^{\frac{1}{2}} \\
&+C(\kappa) |h|^{2s-t} \brac{\brac{\mathcal{M}|G|^2(x_0)}^{\frac{1}{2}} +\brac{\mathcal{M}|G|^2(x_0+h)}^{\frac{1}{2}}+\brac{\mathcal{M}|G|^2(x_0-h)}^{\frac{1}{2}}}\\
 \end{split}
\]
This concludes the proof of \eqref{eq:280523:goal}.
\end{proof}

With the same arguments (actually this is easier), we also find
\begin{proposition}\label{pr:maxfctest2}
Let $s \in (\frac{1}{2},1)$, $n \geq 2$. For any $\kappa \geq  2$ there exists $C(\kappa,s,n)$ such that if $u \in W^{s,2}_0(\R^n_+)$  solves
\[
 \Dels{s}_{\R^n_+} u = \Ds{t} G \quad \text{in $\R^n_+$}
\]
we have for any $x_0 \in \R^n_+$ and any $r > 0$
\begin{equation}\label{eq:280523:goal2}
\begin{split}
 \mathcal{M}^\#_+u(x_0) \aleq& \kappa^{-\alpha} \brac{\mathcal{M}_+|u|^2}^{\frac{1}{2}}(x_0) \\
   & +C(\kappa) \, \brac{\mathcal{M} |G|^2(x_0)}^{\frac{1}{2}}+
 \end{split}
\end{equation}
The constants in $\aleq$ depend only on $n$, $s$ (in particular they are independent on $f$ and $\kappa$).
\end{proposition}

\section{Tangential improvement -- Proof of Theorem~\ref{th:tangential}}\label{s:tangential}
Since by assumption $u \in L^2(\R^n)$, by Sobolev embedding $u \in L^p(\R^n)$ for any $p \in [2,\frac{2n}{n-2s}]$. Since $n \geq 2$ and $s < 1$, $\frac{2n}{n-2s} < \infty$.

For $p \in (2,\frac{2n}{n-2s}]$ we can integrate the inequality \eqref{eq:280523:goal} from \Cref{pr:maxfctest} in $L^p(\R^n)$. Since $\delta_{2,h} u  \in L^p(\R^n)$, we can apply the Fefferman-Stein inequality \Cref{th:feffermanstein} and the maximal theorem (here we need $p \geq 2$) to obtain 

\[
\|\delta_{2,h}u \|_{L^p(\R^n_+)} \aleq \kappa^{-\alpha} \|\delta_{2,h}u\|_{L^p(\R^n_+)} +C(\kappa) |h|^{2s-t}\, \|G\|_{L^p(\R^n_+)}.
\]
Taking $\kappa$ suitably large, we find 
\[
\|\delta_{2,h}u \|_{L^p(\R^n_+)} \aleq |h|^{2s-t}\, \|G\|_{L^p(\R^n_+)}.
\]
Similarly, the estimate from \Cref{pr:maxfctest2} implies
\[
\|u \|_{L^p(\R^n_+)} \aleq \|G\|_{L^p(\R^n_+)}.
\]

In view of \Cref{la:del2todel1} we conclude that for any $\tilde{t} > t$
\[
\sup_{h \in \R^{n-1}\times \{0\}} |h|^{t-2s} \|\delta_{h} u \|_{L^p(\R^n_+)} \aleq \|G\|_{L^p(\R^n_+)}.
\]

In view of \Cref{la:tangentialest} we conclude for any $\tilde{t}_2 > \tilde{t}$, $2s - 1 < \tilde{t}_2$
\[
\begin{split}
 &\brac{\int_0^\infty \int_{\R^{n-1}} \int_{\R^{n-1}} \frac{|u(x',{x_n})-u(y',{x_n})|^p}{|x'-y'|^{n-1+(2s-\tilde{t})p}} dx'\, dy' \, dx_n}^{\frac{1}{p}} \\
 =&  c\brac{\int_{\R^n_+} \int_{\R^n_+} \frac{|u(x',{x_n})-u(y',{x_n})|^p}{|x-y|^{n+(2s-\tilde{t})p}} dx' dx_n dy' dy_n}^{\frac{1}{p}} \\
 \aleq& \|G\|_{L^p(\R^n_+)}.
 \end{split}
\]
\section{The normal direction: Proof of Theorem~\ref{th:normal}}\label{s:normal}
Using the notation $\R^n \ni x = (x',x_n)$, $x' \in \R^{n-1}$, $x_n \in \R$, we set 
\[
\begin{split}
 [f]_{W^{t,p}_T(\R^n_+)} :=& \brac{\int_{\R^n_+}\int_{\R^n_+} \frac{|f(x',x_n)-f(y',x_n)|^{p}}{|x-y|^{n+2s}}\, dx\, dy}^{\frac{1}{p}}\\
 \aeq& \brac{\int_{\R^{n-1}}\int_{\R^n_+} \frac{|f(x',x_n)-f(y',x_n)|^{p}}{|x'-y'|^{n-1+2s}}\, dx\, dy'}^{\frac{1}{p}}
 \end{split}
\]
and
\[
\begin{split}
 [f]_{W^{t,p}_N(\R^n_+)} :=& \brac{\int_{\R^n_+}\int_{\R^n_+} \frac{|f(x',x_n)-f(x',y_n)|^{p}}{|x-y|^{n+2s}}\, dx\, dy}^{\frac{1}{p}}\\
 \aeq& \brac{\int_{0}^\infty \int_{\R^n_+} \frac{|f(x',x_n)-f(x',y_n)|^{p}}{|x_n-y_n|^{1+2s}}\, dx\, dy_n}^{\frac{1}{p}}
 \end{split}
\]
By triangular inequality we have 
\[
 [f]_{W^{t,p}(\R^n_+)} \aleq [f]_{W^{t,p}_N(\R^n_+)} + [f]_{W^{t,p}_T(\R^n_+)} 
\]
If $f$ can be written as $f(x) = \eta(x')\psi(x_n)$, then we find
\[
 [f]_{W^{t,p}(\R^n_+)}  \aleq \|\psi\|_{L^p((0,\infty))}\, [\eta]_{W^{s,p}(\R^{n-1})} + \|\psi\|_{W^{s,p}((0,\infty)}\, \|\eta\|_{L^p(\R^{n-1})}
\]
Assume now that 
\[
 u \in W^{s,2}_0(\R^n_+) \cap W^{\tilde{s},p}_T(\R^n_+)
\]
solves 
\[
 \Dels{s} u = g \quad \text{in $\R^{n}_+$}.
\]
Let $\varphi \in C_c^\infty(\R^n_+)$ and test the equation for $u$ with $\varphi$ to obtain
\[
 \int_{\R^n_+} \int_{\R^n_+} \frac{(u(x',x_n)-u(y',x_n))\, (\varphi(x)-\varphi(y))}{|x-y|^{n+2s}}\, dx\, dy = \tilde{G}[\varphi],
\]
where 
\begin{equation}\label{eq:newpde} 
\int_{\R^n_+} \int_{\R^n_+} \frac{(u(y',x_n)-u(y',y_n))\, (\varphi(x)-\varphi(y))}{|x-y|^{n+2s}}\, dx\, dy = \tilde{G}[\varphi]
\end{equation}
where 
\[
  \tilde{G}[\varphi] = g[\varphi] + \int_{\R^n_+} \int_{\R^n_+} \frac{(u(x',x_n)-u(y',x_n))\, (\varphi(x)-\varphi(y))}{|x-y|^{n+2s}}\, dx\, dy.
\]
Observe we then have, if $0 \leq t < 2s-\tilde{s}$,
\[
 |\tilde{G}[\varphi]| \aleq \brac{\|g\|_{H^{-t,p}(\R^n)} + [u]_{W_T^{\tilde{s},p}(\R^n_+)}} [\varphi]_{W^{2s-\tilde{s},p'}(\R^n)}.
\]
Fix $\eta \in C_c^\infty(\R^{n-1})$  and set
\[
 v(x_n) := \int_{\R^{n-1}} u(z',x_n)\, dz',
\]
and for $\psi \in C_c^\infty((0,\infty))$ set
\[
\tilde{H}(\psi) := \tilde{G}(\tilde{\psi}),
\]
where we set 
\[
 \tilde{\psi}(x) := \psi(x_n) \eta(x')
\]
Then \eqref{eq:newpde} becomes 
\[
 \Dels{s}_{(0,\infty)} v[\psi] = \tilde{H}(\psi) \quad \forall \psi \in C_c^\infty((0,\infty)),
\]
that is 
\[
  \Dels{s}_{(0,\infty)} v = \tilde{H} \quad \text{in $(0,\infty)$}.
\]
Observe that as long as $2s-\tilde{s} > \frac{1}{2}$, using \Cref{la:Ws0halfspace},
\[
\begin{split}
 |\tilde{H}(\psi)| \aleq& \brac{\|g\|_{H^{-t,p}(\R^n)} + [u]_{W_T^{\tilde{s},p}(\R^n_+)}} [\eta \psi]_{W^{2s-\tilde{s},p'}(\R^n)}\\
   \aleq&\brac{\|g\|_{H^{-t,p}(\R^n)} + [u]_{W_T^{\tilde{s},p}(\R^n_+)}} \brac{\|\eta\|_{L^{p'}(\R^n)} [\psi]_{W^{2s-\tilde{s},p'}((0,\infty))} + [\eta]_{W^{2s-\tilde{s},p'}(\R^{n-1})\, \|\psi\|_{L^{p'}((0,\infty))}}}.
 \end{split}
\]
That is 
\[
 \tilde{H} \in \brac{W^{2s-\tilde{s},p'}_0((0,\infty))}^\ast =W^{\tilde{s}-2s,p}((0,\infty)).
\]

\bibliographystyle{abbrv}%
\bibliography{bib}%

\end{document}